\documentclass[leqno,twoside, 12pt]{amsart}
\usepackage{amssymb, latexsym, esint}
\usepackage{color}

\usepackage{amsthm}
\usepackage{amssymb,amsfonts}
\usepackage[all,arc]{xy}
\usepackage{enumerate}
\usepackage{mathrsfs}
\usepackage{amsmath}
\usepackage{verbatim}

\usepackage[left=2.1cm,right=2.1cm,top=2.9cm,bottom=2.9cm]{geometry}

\theoremstyle{plain}
\numberwithin{equation}{section} \numberwithin{figure}{section}

\newtheorem{theorem}{Theorem}[section]
\newtheorem{lemma}[theorem]{Lemma}

  \usepackage{epsfig} 
\usepackage{graphicx}  \usepackage{epstopdf}
\theoremstyle{definition}
\newtheorem{remark}[theorem]{Remark}

\numberwithin{equation}{section}

\usepackage[bookmarksnumbered,colorlinks, linkcolor=blue, citecolor=red, pagebackref, bookmarks, breaklinks]{hyperref}

\begin{document}


\title{A mixed local-nonlocal H\'enon problem in $\mathbb{R}^N$}

\author{Pablo Ochoa}

\address{P. Ochoa. \newline Universidad Nacional de Cuyo, Fac. de Ingenier\'ia. CONICET. Universidad J. A. Maza\\Parque Gral. San Mart\'in 5500\\
Mendoza, Argentina.}
\email{pablo.ochoa@ingenieria.uncuyo.edu.ar}

\author{Ariel Salort}
\address{A. Salort \newline Departamento de Matem\'aticas y Ciencia de Datos, Universidad San Pablo-CEU, CEU Universities, Urbanizaci\'on Montepr\'incipe, 28660 Boadilla del Monte, Madrid, Spain. }
\email{\tt ariel.salort@ceu.es}

\parskip 3pt

\begin{abstract}
In this article, we study a H\'enon-type equation in $\mathbb{R}^N$ driven by a nonlinear operator given by the combination of a local and a nonlocal term. This equation was originally proposed to model spherically symmetric stellar clusters. Here, we prove that, under a suitable relation among the parameters,  there exists a threshold separating  the existence and non-existence of solutions. Moreover, we establish  regularity properties of the solutions.
\end{abstract}
\maketitle

\section{Introduction}

The classical H\'enon problem, introduced in \cite{Henon}, consists of the following problem
\begin{equation} \label{intro.henon}
-\Delta u = |x|^\alpha u^{q-1} \text{ in } B\subset \mathbb{R}^N, \qquad u>0,
\end{equation}
with Dirichlet boundary conditions, where $B$ denotes the unit ball in $\mathbb{R}^N$.  The motivation in H\'enon's original paper \cite{Henon} was a model of stationary solutions of the nonlinear Poisson equation describing cluster density under gravitational interaction with inhomogeneous distribution. In this direction, the seminal work of Ni \cite{Ni} established that the weight $|x|^\alpha$  allows for a wider admissible range of exponents $2<q<2^*_\alpha:=\frac{2(N+\alpha)}{N-2}$ for which solutions exist. This problem started a huge research on elliptic equations with nonlinear right-hand sides. The literature on the subject is now vast, and providing an exhaustive list of contributions lies beyond the scope of this work. In what follows, we restrict ourselves to highlighting a selection of results that extend the H\'enon equation to nonlocal operators or in unbounded domains.

A generalization of \eqref{intro.henon}, involving a mixed local–nonlocal operator, was introduced in \cite{SV}, where existence (and nonexistence) results for positive solutions of
$$
\gamma (-\Delta_p) u+ (1-\gamma)(-\Delta_p)^s u = |x|^\alpha u^{q-1} \text{ in } B\subset \mathbb{R}^N, \qquad u>0,
$$
under exterior Dirichlet boundary conditions were established for an appropriate range of para\-meters $p<N$ and $p<q<p^*_{s,\alpha}:=\frac{p(N+\alpha)}{N-sp}$. Here, for $p>1$ and for a function $u$ smooth enough, $\Delta_p u:= \text{div}(|\nabla u|^{p-2}\nabla u)$ denotes the usual $p-$Laplacian of $u$, and given $s\in(0,1)$, the fractional $p-$Laplacian of $u$ is defined, up to a normalization constant, as 
$$
(-\Delta_p)^s u(x):=\text{p.v.}\int_{\mathbb{R}^N} \frac{|u(x)-u(y)|^{p-2}(u(x)-u(y))}{|x-y|^{N+sp}}\,dy.
$$

Further generalizations of \eqref{intro.henon} to the fractional Laplacian have been obtained in the unbounded setting, where the ball is replaced by the whole space $\mathbb{R}^N$. In \cite{BQ} it was shown that, in contrast to the case of a bounded domain, the equation
$$
(-\Delta)^s u = |x|^\alpha u^{q-1} \text{ in } \mathbb{R}^N, \qquad u>0, 
$$
admits no positive solutions in the range $~2<q<2^*_{s,\alpha}:=\frac{2(N+\alpha)}{N-2s}$, where $\alpha>-2s$, due to the loss of mass at infinity. Similar nonexistence results for the spectral fractional Laplacian were obtained in \cite{SW}.

In order to obtain existence of solutions in   $\mathbb{R}^N$, a confining potential can be added to prevent the loss of mass at infinity. In the linear local case, \cite{Sin} studied the H\'enon-type equation
\begin{equation} \label{intro.eq.1}
-\Delta u + |x|^\beta u^{p-1} = |x|^\alpha u^{q-1} \text{ in } \mathbb{R}^N, \qquad u>0,
\end{equation}
and proved the existence of a positive radial solution for $N\geq 3$ and $p>1$, under the assumptions $\max\{2,p\}<q< p^*_\alpha:=\frac{p(N+\alpha)}{N-2}$ and $\alpha-\beta<\frac{\beta+2(N-1)}{p+2}(q-p)$.

A nonlocal version of \eqref{intro.eq.1} was studied in \cite{Ma} for the fractional Laplacian with $\beta=2$, namely, 
\begin{equation} \label{intro.eq.2}
(-\Delta)^s u + |x|^2 u = |x|^\alpha u^{q-1} \text{ in } \mathbb{R}^N, \qquad u>0.
\end{equation}
In that work, it was shown that for $\frac12<s<\frac{N}{2}$ and $q>1$, under the condition $-N+(\frac{N}{2}-s)q<\alpha<2+(q-2)(\frac{N}{2}-s)$, there exists a positive radially symmetric solution to \eqref{intro.eq.1} that decays to zero at infinity. We note, however, that the technique used in \cite{Ma} requires
 $\beta=2$, and moreover, the resulting range for $\alpha$ does not coincide with the corresponding formula in the local case.
 
 A similar result was proved in \cite[Corollary 5.5]{PdN}, where existence of a radial weak solution of \eqref{intro.eq.2} is guaranteed when $\frac{N}{2}<s$,  $q<\frac{2(N+\alpha)}{N-2s}$ and $\alpha+ \frac{q(1-2s)-2}{2s}<\frac{s}{2}(q-2)(N-1)$. Moreover, the author in \cite{PdN} gives a condition for the existence of solution when the right hand side $|x|^2 u$ is replaced  by $|x|^\beta u^{p-1}$, namely, $\beta(q-2-2sq) + \alpha (2sp -p+2)<2s(p-q)(N-1)$ for $\max\{2,p\}<q<\frac{2(N+\alpha)}{N-2s}$.

\medskip

The goal of this article is to extend existence results for H\'enon-type equations in $\mathbb{R}^N$ to operators exhibiting both local and nonlocal behavior, namely,
\begin{equation} \label{eq}
\begin{cases}
\mathcal{L}_{s,p,\gamma}\, u + |x|^\beta |u|^{p-2}u=|x|^\alpha |u|^{q-2}u \quad \text{in }\mathbb{R}^N, \\
u >0,
\end{cases}
\end{equation}
where $u$ belongs to a suitable functional space of radial functions $\mathcal{W}^{s,p}_\beta(\mathbb{R}^N)$ defined in Section \ref{sec.1} (see \eqref{space}). Given $\gamma\in [0, 1]$, $s\in (0, 1]$ (depending on $\gamma$) and $p>1$, the operator $\mathcal{L}_{s,p,\gamma}$ is given by 
\begin{equation*}
\mathcal{L}_{s,p,\gamma}\, u := \begin{cases}
(-\Delta_p) u, \quad & \text{if}\quad \gamma=1,\\
\gamma (-\Delta_p) u+ (1-\gamma)(-\Delta_p)^s u, \quad & \text{if}\quad \gamma\in (0,1), \, s\in (0,1),\\(-\Delta_p)^s u, \quad & \text{if}\quad \gamma=0,  \, s\in (0,1).
\end{cases}
\end{equation*}

Hence, this operator is a mixed local–nonlocal operator, reducing to the classical 
$p-$Laplacian when  $\gamma=1$, and to the fractional $p-$Laplacian when $\gamma=0$. We will usually take $s\in (0, 1]$, identifying the case $\gamma=1$ with $s=1$. In the linear case (that is, $p=2$) the operator $\mathcal{L}_{s,p,\gamma}$ admits a probabilistic interpretation: the local component corresponds to the continuous part of the dynamics, while the nonlocal component models jump behavior.

The existence and qualitative properties of solutions for mixed local–nonlocal operators with nonlinear terms is an active area of research. Related results on bounded domains can be found, for instance, in \cite{BDVV, DSFV,  GL, ZHT, MMV,  SVWZ}.

Our primary goal is to study the existence of weak solutions to \eqref{eq}, defined variationally via the corresponding energy functional.
\begin{equation}
\mathcal{J}(u)=\frac{\gamma}{p}\|\nabla u\|_p^p + \frac{1-\gamma}{p}  [u]_{s,p}^p + \frac{1}{p}\|u\|_{p,\beta}^p- \frac{1}{q}\|u\|_{q, \alpha}^q
\end{equation}
defined on a suitable space  $\mathcal{W}^{s,p}_\beta(\mathbb{R}^N)$, where $\|u\|_{\beta,p}^p$ denotes the integral $\int_{\mathbb{R}^N} |x|^\beta |u|^p\,dx$, and $[u]_{s,p}$ is the so-called Gagliardo seminorm (see Section \ref{sec.1} for precise definitions). Indeed, $\mathcal{J}$ is a $C^1$ functional whose Fr\'echet derivative is given by
\begin{align*}
\langle \mathcal{J}'(u),v \rangle &=\gamma \int_{\mathbb{R}^N} \Phi_p(\nabla u)\cdot \nabla v\,dx + (1-\gamma)  \int_{\mathbb{R}^N}\int_{\mathbb{R}^N} \frac{\Phi_p(u(x)-u(y))(v(x)-v(y))}{|x-y|^{N+sp}}\,dxdy\\ &\quad +\int_{\mathbb{R}^N} |x|^\beta \Phi_p(u)v\,dx - \int_{\mathbb{R}^N} |x|^\alpha \Phi_q(u)v\,dx \qquad \forall u,v \in \mathcal{W}^{s,p}_\beta(\mathbb{R}^N), 
\end{align*}
where $\Phi_p(t):=|t|^{p-2}t$. Therefore, any critical point of $\mathcal{J}$ is a weak solution of \eqref{eq}.

Our first result establishes a range of parameters for which a weak solution of \eqref{eq} exists.   For further reference, we let for $\gamma\in [0, 1]$,
\begin{equation*}
p^{*}_{\gamma, s, \alpha}:= \begin{cases}
\frac{p(N+\alpha)}{N-sp}, \quad \text{if }\gamma = 0,\\
\frac{p(N+\alpha)}{N-p}, \quad \text{if }\gamma \in (0, 1].
\end{cases}
\end{equation*}

 Precisely, we have:

\begin{theorem} \label{teo.1}
Let $s\in (0,1]$ such that $1/p<s<N/p$,  $\alpha>-sp $ and $p<q<p^{*}_{\gamma, s, \alpha}$. Moreover, assume that
\begin{equation}\label{cond exponent 2}
\alpha-\beta +(q-p)\left(\dfrac{1-N}{p} \right)<0. 
\end{equation}
Then, problem \eqref{eq} admits a radial weak solution. 
\end{theorem}

The proof of Theorem \ref{teo.1} relies mainly on establishing a suitable compact embedding of the fractional Sobolev space of weighted radial functions into an appropriate weighted Lebesgue space (see Theorem \ref{compact}). This compactness result is the key ingredient that enables the application of a Mountain Pass–type existence theorem (see Lemma \ref{lema.BN}).

We note that our result extends the existence obtained in \cite{PdN, Sin,Ma} to the setting of local-nonlocal nonlinear operators on unbounded domains. Furthermore, the range of parameters given in \eqref{cond exponent 2} improves upon that obtained in \cite{Ma} in the particular case  $p=2$ and $\beta=2$.

\medskip
Regarding regularity, in the following result, by employing a De Giorgi's iteration scheme, we find that radial weak solutions of \eqref{eq} are indeed bounded.

\begin{theorem} \label{teo_boundedness}

Let $s\in (0,1]$ be such that $1/p<s<N/p$,  $\alpha>-sp $ and $p<q<p^{*}_{\gamma, s, \alpha}$. Moreover, assume condition \eqref{cond exponent 2}. Then, for any radial weak   solution $u\in \mathcal{W}^{s,p}_\beta(\mathbb{R}^N)$ to \eqref{eq}, there is a constant $C>0$ such that
$$\|u\|_{L^{\infty}(\mathbb{R}^N)}\leq C.$$
\end{theorem}

Finally, by applying a generalization of Pohozaev’s existence result established in \cite{ROS}, we show that the value $\frac{p(N+\alpha)}{N-sp}$ serves as a threshold for the existence of bounded solutions with a weighted control of the gradient.

\begin{theorem} \label{teo.no.exist}

Let $s\in (0,1]$ be such that $1/p<s<N/p$, $\alpha>-sp$ and  $\beta>-sp$. Then, for all $q>p^{*}_{\gamma, s, \alpha}$, problem \eqref{eq} has no solutions $u\in \mathcal{\tilde X}(\mathbb{R}^N)$, being
$$
\mathcal{\tilde X}(\mathbb{R}^N) =\{ u \in \mathcal{W}^{s,p}_\beta(\mathbb{R}^N)\cap L^\infty(\mathbb{R}^N) \text{ such that } |\nabla u(x)||x| \in L^r(\mathbb{R}^N) \text{ for some } r>1\}.
$$
\end{theorem}

\bigskip

The paper is organized as follows. In Section \ref{sec.1} we provide the functional framework and preliminary definitions. In the next Section \ref{sec.2}, we quote the main properties of radial functions related to our problem,    specially, Strauss' type lemmas, generalized Rother's Lemma and  compactness of embeddings between fractional Sobolev spaces of radial functions and appropriate Lebesgue weighted spaces. In Section \ref{sec.3}, we state the existence of solutions to problem \eqref{eq} through variational methods, and in Section \ref{B} we prove the boundedness of radial solutions. Finally, Section \ref{sec.6} is devoted to prove our non-existence result. For convenience of the reader, we add an Appendix with the Brezis-Nirenberg version of the Mountain Pass Theorem  that we employ in the paper.

\section{Preliminaries} \label{sec.1}

\subsection{Functional framework}

Given $p>1$, $\beta>0$ and $s\in(0,1]$  we consider the weighted fractional Sobolev space
$$
W_{\beta}^{s, p}(\mathbb{R}^N)= 
 \{ u \in L^p_\beta(\mathbb{R}^N) \colon [u]_{s,p}<\infty\}
$$
endowed with the norm  
$$
\|u\|_{s, p, \beta}:= \|u\|_{p,\beta} + [u]_{s, p},
$$
where $\|u\|_{p, \beta}:=\left(\int_{\mathbb{R}^N} |x|^\beta|u|^p\,dx\right)^{1/p}$ and the seminorm is given by
\begin{align*}
[u]_{s, p}=
\begin{cases}
\displaystyle\left(\int_{\mathbb{R}^N}\int_{\mathbb{R}^N}\dfrac{|u(x)-u(y)|^p}{|x-y|^{sp+N}}\,dx\,dy \right)^{1/p} &\text{ if } s\in (0,1)\\
\displaystyle\left(\int_{\mathbb{R}^N} |\nabla u|^p\,dx\right)^{1/p} & \text{ if } s=1.
\end{cases} 
\end{align*}
Here, the following weighted Lebesgue spaces are used
$$L^p_\beta(\mathbb{R}^N):= \left\lbrace u:\mathbb{R}^N\to \mathbb{R}: \int_{\mathbb{R}^N}|x|^\beta|u|^p\,dx<\infty\right\rbrace$$
with the corresponding norm $\|u\|_{p, \beta}$.
When $\beta=0$ we just write $\|u\|_{s,p}$ and $\|u\|_p$ instead of $\|u\|_{s,p,0}$ and $\|u\|_{p,0}$, respectively. Moreover, we write $W^{s,p}(\mathbb{R}^N)$ in place of $W^{s,p}_0(\mathbb{R}^N)$.
\\The corresponding subset of radial functions is defined as
$$W_{rad, \beta}^{s, p}(\mathbb{R}^N)= \left\lbrace u \in W_{\beta}^{s,p}(\mathbb{R}^N): \text{ $u$ is radial} \right\rbrace.
$$

With the aforementioned definitions,  the natural framework for introducing weak solutions of \eqref{eq} is the space
\begin{align}\label{space}
\mathcal{W}^{s,p}_\beta(\mathbb{R}^N)=
\begin{cases}
W^{s,p}_{rad,\beta}(\mathbb{R}^N) & \text{ if } \gamma=0\\
W^{1,p}_{rad,\beta}(\mathbb{R}^N)  & \text{ if } \gamma\in(0,1] \text{ and } s\in (0,1).
\end{cases}
\end{align}
endowed with 
$\|\cdot\|_{s,p,\beta}$ when $\gamma=0$ and $s\in (0, 1)$, and $\|\cdot\|_{1,p,\beta}$ when $\gamma \in (0, 1]$.


Once the functional spaces have been introduced,  the relation between the parameter $\gamma \in [0,1]$ in the H\'enon equation \eqref{eq} and the fractional parameter $s\in (0,1]$ becomes clear: 
\begin{itemize}
	\item[(i)] Purely nonlocal case: $\gamma=0$ and $s\in (0,1)$;
	\item[(ii)] Purely local case: $\gamma=1$ (identified with $s=1$);
	\item[(iii)] Mixed local-nonlocal case: $\gamma\in (0,1)$ and $s\in (0,1)$.
\end{itemize}

We also denote the critical Sobolev exponent as $p^*_s := \frac{Np}{N-sp}~$ when $s\in (0,1]$. We will also use the notation  $p^*:=p^*_1$ in the particular case of $s=1$.

\subsection{The operator $\mathcal{L}_{s,p,\gamma}$}

Given $\gamma\in [0,1]$, $s\in (0,1]$, depending on $\gamma$ as follows, and $1<p<q$, we define the mixed local-nonlocal operator as

\begin{equation*}
\mathcal{L}_{s,p,\gamma}\, u := \begin{cases}
(-\Delta_p) u, \quad & \text{if}\quad \gamma=1 \text{ (identified with $s=1$)},\\
\gamma (-\Delta_p) u+ (1-\gamma)(-\Delta_p)^s u, \quad & \text{if}\quad \gamma\in (0,1), \, s\in (0,1),\\(-\Delta_p)^s u, \quad & \text{if}\quad \gamma=0,  \, s\in (0,1).
\end{cases}
\end{equation*}

We shall write $\mathcal{L}_{s,p,\gamma}$ more succinctly as
$$
\mathcal{L}_{s,p,\gamma} \,u :=\gamma (-\Delta)_p u+ (1-\gamma)(-\Delta)_p^s u 
$$
for all $u\in \mathcal{W}^{s,p}_\beta(\mathbb{R}^N)$. Here, we denote the fractional $p-$Laplacian of order $s\in(0,1]$ as
\begin{align*}
(-\Delta_p)^s u(x) :=
\begin{cases}
 \text{p.v.}\displaystyle  \int_{\mathbb{R}^N}  \frac{|u(x)-u(y)|^{p-2}(u(x)-u(y))}{|x-y|^{N+sp}} \,dy & \text{ if } s\in (0,1)\\[15pt]
-\text{div}\,(|\nabla u(x)|^{p-2}\nabla u(x)) & \text{ if } s=1.
\end{cases}
\end{align*}
Thus, $(-\Delta_p)^1 u$ stands for $-\Delta_p u$.

We point out that since $W^{1, p}(\mathbb{R}^n)\subset W^{s, p}(\mathbb{R}^n)$ for all $s\in (0, 1)$, the following representation formula holds: given $u\in \mathcal{W}^{s,p}_\beta(\mathbb{R}^N)$,
\begin{align}\label{representation}
\langle \mathcal{L}_{s,p,\gamma} \,u ,v \rangle &=\gamma \int_{\mathbb{R}^N} \Phi_p(\nabla u)\cdot \nabla v\,dx + (1-\gamma)  \int_{\mathbb{R}^N}\int_{\mathbb{R}^N} \frac{\Phi_p(u(x)-u(y))(v(x)-v(y))}{|x-y|^{N+sp}}\,dxdy 
\end{align} 
for all $v\in \mathcal{W}^{s,p}_\beta(\mathbb{R}^N)$, being $\Phi_p(t):=|t|^{p-2}t$ for $t\in \mathbb{R}$.  Therefore,  we say that $u\in \mathcal{W}^{s,p}_\beta(\mathbb{R}^N)$ is a \emph{weak solution} of \eqref{eq} if
$$
\langle \mathcal{L}_{s,p,\gamma}u,v \rangle  + \int_{\mathbb{R}^N} |x|^\beta \Phi_p(u)v\,dx  = \int_{\mathbb{R}^N} |x|^\alpha \Phi_q(u)v\,dx \qquad \forall v\in \mathcal{W}^{s,p}_\beta(\mathbb{R}^N).
$$
Observe that weak solutions are critical points of the functional 
\begin{equation*}
\mathcal{J}(u):=\frac{\gamma}{p}\|\nabla u\|_p^p + \frac{1-\gamma}{p}  [u]_{s,p}^p + \frac{1}{p}\|u\|_{p,\beta}^p- \frac{1}{q}\|u\|_{q, \alpha}^q
\end{equation*}
defined on $\mathcal{W}^{s,p}_\beta(\mathbb{R}^N)$, whose Fr\'echet derivative is given by 
\begin{align*}
\langle \mathcal{J}'(u),v \rangle &= \langle \mathcal{L}_{s,p,\gamma} u,v \rangle  +\int_{\mathbb{R}^N} |x|^\beta \Phi_p(u)v\,dx - \int_{\mathbb{R}^N} |x|^\alpha \Phi_q(u)v\,dx \qquad \forall v\in \mathcal{W}^{s,p}_\beta(\mathbb{R}^N).
\end{align*}

\section{Behavior of functions in $W^{s,p}_{rad, \beta}(\mathbb{R}^N)$}\label{sec.2}

The Strauss Radial Lemma is a powerful tool that quantifies the decay of radial functions at infinity. The following version is derived  from Theorems 10 and 13 in \cite{SS}.

\begin{lemma}[Strauss' Lemma]\label{Strauss lemma}
Assume $1/p < s \leq 1$ and $sp<N$. Then, there is a constant $C>0$ such that for any $u\in W^{s, p}_{rad}(\mathbb{R}^N)$, 
\begin{equation}
|u(x)|\leq C|x|^{(1-N)/p}\|u\|_{s, p},
\end{equation}for all $x\in \mathbb{R}^N$.
\end{lemma}~

\begin{remark}\label{cont with beta}We observe that in the previous lemma, we may replace the norm $\|u\|_{s, p}$ with the norm $\|u\|_{s, p, \beta}$. Indeed, given $s\in (0, 1]$ and $u\in W_{\beta}^{s, p}(\mathbb{R}^N)$ with compact support,
\begin{equation}
\begin{split}
\|u\|^p_{p}&= \int_{|x|<1}|u|^p\,dx + \int_{|x|\geq 1}|u|^p\,dx \\ &
\leq \left( \int_{|x|<1}|u|^{p^*_s}\,dx\right)^{p/p_s^{*}}\left(|B| \right)^{N/sp}+  \int_{|x|\geq 1}|x|^\beta|u|^p\,dx\\& \leq 
C\|u\|^p_{p_s^{*}}+\|u\|^p_{p,\beta}\\& \leq C[u]_{s, p}^p +\|u\|^p_{p,\beta} \leq  C\|u\|^p_{s, p, \beta},
\end{split}
\end{equation}
where, for $s\in (0,1]$  we have used the inequality (see for instance, \cite[Theorem 6.5]{NPV})
$$
\|u\|_{p_s^{*}} \leq C[u]_{s, p}
$$
being $C$ a positive constant depending of $N$, $s$ and $p$.

By density, this proves that $\|u\|_{s,p}\leq C\|u\|_{s, p, \beta}$ in $W_{ \beta}^{s, p}(\mathbb{R}^N)$ and the continuity of the embedding $W_{ \beta}^{s, p}(\mathbb{R}^N) \subset W^{s, p}(\mathbb{R}^N)$.  
\end{remark}

\begin{remark}
Recall that (see for instance \cite[Lemma 4.2]{FBS}), for $u\in W^{1,p}(\mathbb{R}^N)$, $1\leq p <\infty$ and $s\in (0,1)$, it holds that
\begin{equation} \label{cota.local.nolocal}
[u]_{s,p}^p \leq \frac{N\omega_N}{p} \left(\frac{2^p}{s}\|u\|_p^p + \frac{1}{1-s} \|\nabla u\|_{p}^p\right).
\end{equation}
Then, by \eqref{cota.local.nolocal}  and Remark \ref{cont with beta}, 
\begin{equation} \label{cota.lnl}
[u]_{s,p}^p \leq C(\|u\|_p^p + \|\nabla u\|_{p}^p) \leq C(\|u\|_{1, p,\beta}^p + \|\nabla u\|_p^p)  \leq C \|u\|_{1,p,\beta}^p.
\end{equation}
 \end{remark}

We quote the following continuity embedding from  \cite[Theorem 6.2]{EP}:

\begin{lemma}\label{imbedding H}
Let $0<s<N/p$, $c>-N $, $(1-sp)c \leq (N-1)sp$ and let $q= \frac{p(N+c)}{N-sp}$. Then, there is $C>0$ such that
\begin{equation}
\left(\int_{\mathbb{R}^N}|x|^c |u|^q\,dx\right)^{1/q} \leq C\|u\|_{H^{s, p}},
\end{equation}for all radial $u\in H^{s, p}(\mathbb{R}^N)$.
\end{lemma}

Here $\|u\|_{H^{s, p}}$ denotes the norm of $u$ in the Bessel potential space $H^{s, p}(\mathbb{R}^N)$  (see, for instance, \cite{AMS} for a precise definition). We next show that we may replace $\|u\|_{H^{s, p}}$ in the previous lemma by the norm $\|u\|_{s, p}$.

\begin{lemma}\label{embedding with W}
Let $s\in (0,1]$ such that $1<sp<N$. Let $\alpha \in \mathbb{R}$ be such that $-sp<\alpha$,  and  $p<q<\frac{p(N+\alpha)}{N-sp}$. Then, there are a constant  $C>0$ and a parameter $c=c(q, s)$ with $-sp<c<\alpha$, such that
\begin{equation}\label{objective}
\left(\int_{\mathbb{R}^N}|x|^c |u|^q\,dx\right)^{1/q} \leq C\|u\|_{{s, p}},
\end{equation}for all  $u\in W_{rad}^{s, p}(\mathbb{R}^N)$.
\end{lemma}

\begin{proof}
Let $s\in (0,1]$ such that $1<sp<N$, and consider $\alpha$ such that  $-sp<\alpha$.  Let $p<q<\frac{p(N+\alpha)}{N-sp}$ and let $u\in W^{s, p}_{rad}(\mathbb{R}^N)$. We choose $0<s'<s$ close enough to $s$ so that it satisfies the following  conditions. First, $s'p>1$ and $s'$ satisfies the relation
$$q<\dfrac{p(N+\alpha)}{N-s'p},$$
which is true since  the function
$$g(r, t):=\frac{p(N+r)}{N-tp}$$is increasing in $r$ and $t$, and $s'$ is close to $s$. 

Now, there is $c=c(q, s)<\alpha$ such that
\begin{equation}\label{cond c 1}
q=\dfrac{p(N+c)}{N-s'p}.
\end{equation} In order to apply Lemma \ref{imbedding H}, we need to check that 
\begin{equation}\label{cond c}
(1-s'p)c\leq (N-1)s'p.
\end{equation}Now, recalling that $1<s'p<N$ and $q>p$, we get
\begin{equation}
\begin{split}
(1-s'p)c  &= (s'p-1)\left(N- \dfrac{q(N-s'p)}{p}\right)\\ & \leq (s'p-1)\left( N -(N-s'p)\right)\\& \leq (N-1)s'p,
\end{split}
\end{equation}which proves \eqref{cond c}. Then, by Lemma \ref{imbedding H}, it follows that
\begin{equation}\label{imb}
\left(\int_{\mathbb{R}^N}|x|^c |v|^q\,dx\right)^{1/q} \leq C\|v\|_{H^{s', p}},
\end{equation}for all radial $v\in H^{s', p}(\mathbb{R}^N)$.  Recall that $H^{s',p}(\mathbb{R}^N)$ denotes the Bessel potential space. Since $s'<s$, by Theorem 11.1 in \cite{AMS}, the following inclusion holds:
$$W^{s, p}(\mathbb{R}^N) \subset H^{s', p}(\mathbb{R}^N) \quad \text{with continuity}.$$
Thus, from \eqref{imb}, we conclude \eqref{objective}. This ends the proof.
\end{proof}

Next, we prove a continuity embedding between the spaces $W^{s,p}_{rad, \beta}(\mathbb{R}^N)$ and $L^q_\alpha(\mathbb{R}^N)$, for any  $p<q<\frac{p(N+\alpha)}{N-sp}$. 

\begin{lemma}\label{continuity}
Let $s \in(0,1]$ such that $1/p<s<N/p$,  $-sp<\alpha$ and $p<q<\frac{p(N+\alpha)}{N-sp}$. Moreover, assume that
\begin{equation}\label{cond exponents}
\alpha-\beta +(q-p)\left(\dfrac{1-N}{p} \right)<0.
\end{equation}
Then, for any $u\in W^{s,p}_{rad, \beta}(\mathbb{R}^N)$, there holds
\begin{equation}\label{ineq with eta}
\int_{\mathbb{R}^N}|x|^\alpha |u|^q\,dx \leq C\|u\|_{s, p}^\eta\|u\|_{p, \beta}^\omega ,
\end{equation}for some $\eta, \omega>0$ depending on $N, p, q$ and $s$.
\end{lemma}

\begin{proof}
Let $\varepsilon >0$. Since $q<\frac{p(N+\alpha)}{N-sp}$, by Lemma \ref{embedding with W}, there is $c<\alpha$ such that
\begin{equation}\label{objective2}
\left(\int_{\mathbb{R}^N}|x|^c |u|^q\,dx\right)^{1/q} \leq C\|u\|_{{s, p}}.
\end{equation}Now,
\begin{equation}\label{ineq 1 eps}
\int_{|x|\leq \varepsilon}|x|^\alpha |u|^q\,dx = \int_{|x|\leq \varepsilon}|x|^{\alpha-c}|x|^c|u|^q\,dx \leq C\varepsilon^{\alpha-c}\|u\|^q_{s, p}.
\end{equation}Next, by Strauss' Lemma \ref{Strauss lemma} and Remark \ref{cont with beta}, we get that
\begin{equation}
\begin{split}
\int_{|x|>\varepsilon}|x|^\alpha |u|^q\,dx &= \int_{|x|>\varepsilon}|x|^{\alpha-\beta}|x|^\beta |u|^{q-p}|u|^p\,dx \\& \leq C\int_{|x|>\varepsilon}|x|^\beta |u|^{p} |x|^{\alpha-\beta}\left(|x|^{\frac{1-N}{p}}\|u\|_{s, p}\right)^{q-p}\,dx \\& = C \|u\|_{s, p}^{q-p} \int_{|x|>\varepsilon}|x|^\beta |u|^{p}  |x|^{\alpha-\beta +(q-p)\frac{1-N}{p}}\,dx .
\end{split}
\end{equation}
Condition \eqref{cond exponents} yields that
\begin{equation}\label{ineq 2 eps}
\int_{|x|>\varepsilon}|x|^\alpha |u|^q\,dx \leq C\varepsilon^{\alpha-\beta +(q-p)\frac{1-N}{p}}\|u\|^p_{p, \beta} \|u\|_{s, p}^{q-p}.
\end{equation}Therefore, combining \eqref{ineq 1 eps} and \eqref{ineq 2 eps},
\begin{equation}
\int_{\mathbb{R}^N}|x|^\alpha |u|^q\,dx \leq C\left(\varepsilon^{\alpha-c}\|u\|^p_{s, p} +  \varepsilon^{\alpha-\beta +(q-p)\frac{1-N}{p}}\|u\|^p_{p, \beta} \|u\|_{s, p}^{q-p} \right).
\end{equation}Minimizing the auxiliary function 
$$g(\varepsilon):= C_1\varepsilon^{e_1}+C_2\varepsilon^{-e_2},$$with
$$C_1=\|u\|^p_{s, p}, \, C_2=\|u\|^p_{p, \beta} \|u\|_{s, p}^{q-p}$$and
\begin{equation}\label{exponents e}
e_1=\alpha-c, \quad -e_2=\alpha-\beta +(q-p)\frac{1-N}{p},
\end{equation}
we finally get \eqref{ineq with eta} with
$$\eta=\dfrac{pe_2+ (q-p)e_1}{e_1+e_2} \quad \text{and}\quad \omega = \dfrac{pe_1}{e_1+e_2}.$$
Observe that $\eta$ is positive since by \eqref{cond exponents},
$$
\eta= \frac{p(\beta-\alpha)+(q-p)(N-1)+(q-p)(\alpha -c)}{e_1+e_2}>\dfrac{(q-p)(\alpha -c)}{e_1+e_2}>0.
$$ 
This concludes the proof.
\end{proof}

We are now ready to prove the compactness of the embedding $W^{s,p}_{rad, \beta}(\mathbb{R}^N)$ into $L^q_\alpha(\mathbb{R}^N)$.

\begin{theorem} \label{compact}
Let $s\in (0,1]$ such that $1/p<s<N/p$,  $\alpha>-sp $ and $p<q<\frac{p(N+\alpha)}{N-sp}$. Moreover, assume that
\begin{equation}\label{cond exponents 2}
\alpha-\beta +(q-p)\left(\dfrac{1-N}{p} \right)<0. 
\end{equation}
Then, the following embedding is compact
$$W^{s,p}_{rad, \beta}(\mathbb{R}^N) \subset L^q_\alpha(\mathbb{R}^N).
$$

\end{theorem}

\begin{proof}
Suppose that $\{u_n\}_{n\in\mathbb{N}}\subset W^{s,p}_{rad, \beta}(\mathbb{R}^N)$ and $u_n \rightharpoonup 0$. Let $M>0$ such that $\|u_n\|_{s, p, \beta}\leq M$. Given $\varepsilon \in (0,1]$, we write
\begin{equation}
\|u_n\|^q_{q, \alpha}= \int_{|x|< \varepsilon}|u_n|^q|x|^{\alpha}\,dx + \int_{ \varepsilon \leq |x|< 1/\varepsilon}|u_n|^q|x|^{\alpha}\,dx + \int_{|x|\geq 1/\varepsilon}|u_n|^q|x|^{\alpha}\,dx := (i)+(ii)+(iii).
\end{equation}
Regarding $(i)$, we first observe that since $q<\frac{p(N+\alpha)}{N-sp}$, then by Lemma \ref{embedding with W} that there is $c \in (-sp,\alpha)$ such that
$$\int_{\mathbb{R}^N}|x|^c|u_n|^q\,dx \leq C\|u_n\|^q_{s, p}\leq C\|u_n\|_{s, p, \beta}^q \leq CM^q,$$where we have used  Remark \ref{cont with beta}. Hence,

\begin{equation}\label{I}
(i)= \int_{|x|<\varepsilon}|x|^{\alpha-c}|x|^c|u_n|^q\,dx \leq C\varepsilon^{\alpha-c}M^q.
\end{equation}

To deal with (ii), we let $\Omega_\varepsilon:=\left\lbrace x\in \mathbb{R}^N: \varepsilon \leq |x|<\varepsilon^{-1}\right\rbrace$. Observe that
$$\int_{\Omega_\varepsilon}|u_n|^p\,dx = \int_{\Omega_\varepsilon}|x|^{-\beta}|x|^\beta|u_n|^p\,dx \leq \varepsilon^{-\beta} \int_{\Omega_\varepsilon}|x|^\beta|u_n|^p\,dx \leq M^p\varepsilon^{-\beta}.$$
By Remark \ref{cont with beta} and the compact embedding $W^{s, p}\subset L^{p}$ in bounded domains (see, for instance \cite[Theorem 7.1]{NPV}), it follows, up to a subsequence that we do not relabel, that
$$u_n\to 0 \quad \text{in }L^p(\Omega_\varepsilon).$$Thus, by Lemma \ref{Strauss lemma}
\begin{equation}\label{II}
(ii) \leq C(\varepsilon)\|u_n\|_{L^{\infty}(\Omega_\varepsilon)}^{q-p}\int_{\Omega_\varepsilon}|u_n|^p\,dx\leq C(\varepsilon,M)\int_{\Omega_\varepsilon}|u_n|^p\,dx \to 0,
\end{equation}
as $n\to \infty$.
Finally, we apply the Strauss' Lemma \ref{Strauss lemma} to get
\begin{equation}
(iii)\leq C \int_{|x|\geq 1/\varepsilon}|x|^\beta|x|^{\alpha-\beta}|u_n|^p\left(|x|^{(1-N)/p}\|u_n\|_{s, p, \beta} \right)^{q-p}\,dx\leq CM^{q-p} \int_{|x|\geq 1/\varepsilon}|x|^\delta|x|^\beta|u_n|^p\,dx,
\end{equation} 
where $\delta:=\alpha-\beta-(p-q)\frac{N-1}{p}<0$ by assumption. Hence
\begin{equation}\label{III}
(iii)\leq CM^{q}\varepsilon^{-\delta}.
\end{equation}

Finally, the conclusion follows combining \eqref{I}, \eqref{II} and \eqref{III}. 
\end{proof}

\section{Existence of solutions}\label{sec.3}

\begin{proof}[Proof of Theorem \ref{teo.1}]
We check the conditions of the Brezis-Nirenberg lemma given in Lemma \ref{lema.BN}. We recall that $E=\mathcal{W}^{s,p}_\beta(\mathbb{R}^N)$ given in \eqref{space}.

\noindent \textbf{(i)}\,  Case $\gamma=0$ and $s\in (0,1)$. By \eqref{space}, the underlying space is $W^{s, p}_{rad, \beta}(\mathbb{R}^N)$. Let $u$ such that $\|u\|_{s,p,\beta}\ll 1$. Then, by Theorem \ref{compact}, 
\begin{align} \label{dd1}
\begin{split}
\mathcal{J}(u)&=\frac{1}{p} [u]_{s,p}^p + \frac{1}{p}\|u\|_{p,\beta}^p  - \frac{1}{q} \|u\|_{q,\alpha}^q\\
&\geq \frac{1}{p} \left( [u]_{s,p}^p + \|u\|_{p,\beta}^p \right)  -C\|u\|_{s,p,\beta}^q\\
&\geq \frac{1}{p2 ^p} \left( [u]_{s,p} + \|u\|_{p,\beta} \right)^p- C \|u\|_{s,p,\beta}^q=
 \frac{1}{p2 ^p} \|u\|_{s,p,\beta}^p- C \|u\|_{s,p,\beta}^q
 \end{split}
\end{align}
where we have used that $a^p + b^p \geq \max\{a^p, b^p\} \geq (\frac{a+b}{2})^p$. 

When $\gamma=1$, by \eqref{space}, the underlying space is now $W^{1, p}_{rad, \beta}(\mathbb{R}^N)$. Analogously, for any $u$ such that $\|u\|_{1,p,\beta}\ll 1$, by Theorem \ref{compact}
\begin{align} \label{dd2}
\begin{split}
\mathcal{J}(u)&=\frac{1}{p}\|\nabla u\|_p^p + \frac{1}{p}\|u\|_{p,\beta}^p- \frac{1}{q}\|u\|_{q,\alpha}^q \geq  \frac{1}{p2 ^p} \|u\|_{1,p,\beta}^p- C \|u\|_{1,p,\beta}^q,
\end{split}
\end{align}
When $\gamma \in (0,1)$ and $s\in (0,1)$, for $u\in W^{1, p}_{rad, \beta}(\mathbb{R}^N)$ such that $\|u\|_{s,p,\beta}\ll 1$,  Theorem \ref{compact}   yields 
\begin{align} \label{dd3}
\begin{split}
\mathcal{J}(u)&=\frac{\gamma}{p}\|\nabla u\|_p^p + \frac{1-\gamma}{p}  [u]_{s,p}^p + \frac{1}{p}\|u\|_{p,\beta}^p- \frac{1}{q}\|u\|_{q,\alpha}^q\\
&\geq \frac{\gamma}{p} \left(  \|\nabla u\|_p^p  + \|u\|_{p,\beta}^p \right)  -C\|u\|_{1,p,\beta}^q\\
&\geq \frac{\gamma}{p2 ^p} \left( \|\nabla u\|_p + \|u\|_{p,\beta} \right)^p- C \|u\|_{1,p,\beta}^q=
 \frac{\gamma}{p2 ^p} \|u\|_{1,p,\beta}^p- C \|u\|_{1,p,\beta}^q,
 \end{split}
\end{align}
Since $p<q$, from \eqref{dd1}, \eqref{dd2} and \eqref{dd3} we have that $\mathcal{J}(u)>0$ for $\|u\|_{s,p,\beta} =r$ small enough.

\medskip

\noindent \textbf{(ii)}\,  Let $\gamma \in [0,1]$ and $s\in (0,1]$ (where we identify $\gamma=1$ with $s=1$). Let $w>0$ fixed, and define $v:=tw$, $t\in\mathbb{R}$. Then
\begin{align*}
\mathcal{J}(v) = 
\frac{\gamma t^p}{p}\|\nabla w\|_p^p+ \frac{(1-\gamma)t^p}{p} [w]_{s,p}^p + \frac{t^p}{p}\|w\|_{p,\beta}^p  - \frac{t^q}{q} \|w\|_{q,\alpha}^q \leq C_1 t^p - C_2 t^q.
\end{align*}
Since $p<q$, there exists $t$ such that $\mathcal{J}(v)<0$ and $\|v\|>r$.

\medskip
 
\noindent \textbf{(iii)}\,  Let $\gamma\in (0,1)$ and $s\in (0,1)$. Let  $\{u_n\}_{n\in\mathbb{N}}\subset W^{1,p}_{rad, \beta}(\mathbb{R}^N)$  be such that $\mathcal{J}'(u_n)\to 0$ and $\mathcal{J}(u_n)\to c$. Since $\mathcal{J}'(u_n)\to 0$, using \eqref{cota.lnl} we get that
\begin{align*}
\left|\langle \mathcal{J}'(u_n),u_n \rangle\right| &= \left|\gamma \|\nabla u_n\|_p^p + (1-\gamma)[u_n]_{s,p}^p +\|u_n\|_{p,\beta}^p - \|u_n\|_{q,\alpha}^q \right| \leq \|u_n\|_{1,p,\beta}
\end{align*}
for $n$ large enough. Since $|\mathcal{J}(u_n)|\leq C$, we have that
$$
\left|\frac{\gamma}{p} \|\nabla u_n\|_p^p + \frac{1-\gamma}{p}[u_n]_{s,p}^p +\frac{1}{p}\|u_n\|_{p,\beta}^p -\frac{1}{q}\|u_n\|_{q,\alpha}^q \right|\leq C.
$$
The last two expressions give that
\begin{align*}
\gamma \|\nabla u_n\|_p^p + (1-\gamma)[u_n]_{s,p}^p &+  \|u_n\|_{p,\beta}^p \leq pC   +\frac{p}{q}\|u_n\|_{q,\alpha}^q\\
&\leq
pC + \frac{p}{q}\gamma \|\nabla u_n\|_p^p +\frac{p}{q} (1-\gamma)[u_n]_{s,p}^p +\frac{p}{q}\|u_n\|_{p,\beta}^p + \frac{p}{q}\|u_n\|_{1,p,\beta}
\end{align*}
from where, since $p<q$, 

\begin{equation} \label{eq.b}
\begin{split}
\left(1-\frac{p}{q}\right) \|u_n\|_{1, p, \beta}^p& \leq C(p, \gamma)\left(1-\frac{p}{q}\right)\left( \gamma \|\nabla u_n\|_p^p + (1-\gamma)[u_n]_{s,p}^p +  \|u_n\|_{p,\beta}^p \right) \\& \leq C(p, \gamma)\left(pC+ \frac{p}{q}\|u_n\|_{1,p,\beta}\right).
\end{split}
\end{equation}

It follows that $\{u_n\}_{n\in\mathbb{N}}$ is bounded in $W^{1,p}_{rad, \beta}(\mathbb{R}^N)$.


 Then, up to a subsequence, $u_n\rightharpoonup u$ weakly in $W^{1,p}_{rad, \beta}(\mathbb{R}^N)$. From the compact embedding given in Theorem \ref{compact}, we get that
$u_n \to u$ strongly in $L^q_\alpha(\mathbb{R}^N)$, which implies

$$
\int_{\mathbb{R}^N} |x|^\alpha |u_n-u|^q\,dx \to 0.
$$
Therefore
\begin{align*}
\Big|\langle \mathcal{L}_{s,p,\gamma}\, u_n,u_n-u \rangle +\int_{\mathbb{R}^N} |x|^\beta |u_n|^{p-2}&u_n (u_n-u)\,dx  \Big| = \left|\langle \mathcal{J}'(u_n), u_n-u \rangle + \int_{\mathbb{R}^N} |x|^\alpha |u_n|^{q-2}u_n(u_n-u)\,dx\right|\\
&\leq
\|\mathcal{J}'(u_n)\|_{W^{1,p}_\beta(\mathbb{R}^N)'} \|u_n-u\|_{1, p, \beta} +
 \int_{\mathbb{R}^N} |x|^\alpha |u_n|^{q-1}|u_n-u|\,dx,
 \end{align*}
where $W^{1,p}_\beta(\mathbb{R}^N)'$ denotes the dual space of $W^{1,p}_\beta(\mathbb{R}^N)$. The first term goes to 0 since $\mathcal{J}'(u_n)\to 0$ and the sequence $\{u_n\}_{n\in \mathbb{N}}$ is bounded in $W^{1,p}_{rad, \beta}(\mathbb{R}^N)$. For the second term, observe that by H\"older's inequality and \eqref{ineq with eta}, it vanishes as $n\to\infty$. Indeed, 
\begin{align*}
\int_{\mathbb{R}^N} |x|^\alpha |u_n|^{q-1}|u_n-u|\,dx &=
\int_{\mathbb{R}^N} |x|^{\frac{\alpha(q-1)}{q}}  |u_n|^{q-1} |x|^\frac{\alpha}{q}|u_n-u|\,dx \\
&\leq 
\left( \int_{\mathbb{R}^N} |x|^{\alpha}  |u_n|^{q}\,dx \right)^\frac{q-1}{q}\left(\int_{\mathbb{R}^N} |x|^\alpha|u_n-u|^q\,dx \right)^\frac{1}{q}\\
&\leq 
C\|u_n\|_{1, p}^{\eta_1}\left(\int_{\mathbb{R}^N}|x|^\beta |u_n|^p\,dx\right)^{\omega_1} \left(\int_{\mathbb{R}^N} |x|^\alpha|u_n-u|^q\,dx \right)^\frac{1}{q}\\
&\leq 
C_1\|u_n\|_{1, p,\beta}^{\eta_2} \left(\int_{\mathbb{R}^N} |x|^\alpha|u_n-u|^q\,dx \right)^\frac{1}{q}\\
&\leq 
C_2 \left(\int_{\mathbb{R}^N} |x|^\alpha|u_n-u|^q\,dx \right)^\frac{1}{q} \to 0\quad  \text{ as } n\to\infty,
\end{align*}
where $C$, $C_1$, $C_2$, $\eta_1$, $\eta_2$ and $\omega_1$ are positive constants depending only of $N,s,p,q$.

Therefore,   the operator $\mathcal{\tilde L}$ defined  as
$$
\langle  \mathcal{\tilde L} u,v\rangle := \langle \mathcal{L}_{s,p,\gamma}\, u,v\rangle + \int_{\mathbb{R}^N} |x|^\beta |u|^{p-2}uv\,dx 
$$
fulfills that 
$$
\langle  \mathcal{\tilde L} u_n,u_n-u \rangle \to 0 \text{ as } n\to \infty.
$$
Hence, by Lemma \ref{rem.S.prop} below, we get that $u_n\to u$ strongly in $W^{1,p}_\beta(\mathbb{R}^N)$.

When $\gamma=0$ and $s\in (0,1)$, and when $\gamma=1$ (identifying this case with $s=1$), the reasoning is analogous.

\medskip

\noindent \textbf{(iv)}\,   Let $\gamma \in (0,1)$ and $s\in (0,1)$. To obtain a positive solution, we choose the mapping $P$ defined as $P(u)=|u|$. Trivially, $P(0)=0$ and $P(v)=v$, where $v$ is given (ii). Moreover, it is direct that $\| \nabla|u|\|_p = \|\nabla u\|_p$, $\||u|\|_p= \|u\|_p$, $\| |u|\|_{p,\beta}=\|u\|_{p,\beta}$ and $\||u|\|_{q,\alpha} = \|u\|_{q,\alpha}$. Furthermore, since $||u(x)|-|u(y)|| \leq |u(x)-u(y)|$, and $t\mapsto t^p$ is monotone increasing for $t\geq 0$, $||u(x)|-|u(y)||^p \leq |u(x)-u(y)|^p$. This gives that $[|u|]_{s,p} \leq [u]_{s,p}$. All these relations give that for any $u\in W^{1,p}_\beta(\mathbb{R}^N)$,
$$
\mathcal{J}(P(u)) \leq \mathcal{J}(u).
$$
An analogous (and simpler) analysis can be made in the cases $\gamma=0$ and $s\in (0,1)$, and $\gamma=1$ (identified with $s=1$).

Therefore, the conclusion follow from  Lemma \ref{lema.BN}. This concludes the proof. 
\end{proof}
 
\begin{lemma}\label{rem.S.prop}Let $u_n \in \mathcal{W}^{s, p}_\beta(\mathbb{R}^N)$ be  a sequence such that $u_n\rightharpoonup u$ and 
$$
\langle  \mathcal{\tilde L} u_n,u_n-u \rangle \to 0 \text{ as } n\to \infty.
$$Then $u_n\to u$ strongly in $\mathcal{W}^{s, p}_\beta(\mathbb{R}^N)$.

\end{lemma}

\begin{proof}Observe that the weak convergence implies that
\begin{equation}\label{weak S}
\langle  \mathcal{\tilde L} u_n-\mathcal{\tilde L}u,u_n-u \rangle \to 0 \text{ as } n\to \infty.
\end{equation}
We will appeal to the following well-known inequalities for $\xi, \eta \in \mathbb{R}^N$:
\begin{equation}\label{ineq degenerate}
(\Phi_p(\xi) - \Phi_p(\eta))(\xi-\eta)\geq C|\xi-\eta|^{p}, \quad p \geq 2,
\end{equation}and
\begin{equation}\label{ineq singular}
(\Phi_p(\xi) - \Phi_p(\eta))(\xi-\eta)\geq (p-1)\dfrac{|\xi -\eta|^2}{(|\xi|+|\eta|)^{2-p}}, \quad 1< p < 2,
\end{equation}where $\Phi_p(\xi)=|\xi|^{p-2}\xi$. First, assume $\gamma=0$ and hence $\mathcal{W}^{s, p}_\beta(\mathbb{R}^N)=W^{s, p}_{rad, \beta}(\mathbb{R}^N).$ Now, if $p\geq 2$, we have
\begin{equation*}
\begin{split}
& \langle  \mathcal{\tilde L} u_n-\mathcal{\tilde L}u,u_n-u \rangle \\& = \int_{\mathbb{R}^N}\int_{\mathbb{R}^N}\dfrac{\left(\Phi_p(u_n(x)-u_n(y))-\Phi_p(u(x)-u(y))\right)(u_n(x)-u_n(y)-(u(x)-u(y)))}{|x-y|^{N+sp}}\,dx\,dy \\& + \int_{\mathbb{R}^N}|x|^\beta\left(\Phi_p(u_n(x))-\Phi_p(u(x))\right)(u_n(x)-u(x))\,dx \\& \geq C\left([u_n-u]^p_{s, p}+\|u_n-u\|_{p, \beta}^p\right)\geq C\|u_n-u \|_{s, p, \beta}^p.
\end{split}
\end{equation*}Hence, \eqref{weak S} gives the strong convergence in $W^{s, p}_\beta(\mathbb{R}^N)$. If now $\gamma\in (0, 1]$, then  $\mathcal{W}^{s, p}_{\beta}(\mathbb{R}^N)=W^{1, p}_{rad, \beta}(\mathbb{R}^N).$ If $p\geq 2$, then by \eqref{ineq degenerate}
\begin{equation}\label{similar i}
\begin{split}
& \langle  \mathcal{\tilde L} u_n-\mathcal{\tilde L}u,u_n-u \rangle \\& = \gamma \int_{\mathbb{R}^N}\left( \Phi_p(\nabla u_n(x))-\Phi_p(\nabla u(x))\right)(\nabla u_n(x)-\nabla u(x))\,dx \\& +(1-\gamma)\int_{\mathbb{R}^N}\int_{\mathbb{R}^N}\dfrac{\left(\Phi_p(u_n(x)-u_n(y))-\Phi_p(u(x)-u(y))\right)(u_n(x)-u_n(y)-(u(x)-u(y)))}{|x-y|^{N+sp}}\,dx\,dy \\& + \int_{\mathbb{R}^N}|x|^\beta\left(\Phi_p(u_n(x))-\Phi_p(u(x))\right)(u_n(x)-u(x))\,dx \\& \geq C\left(\gamma[u_n-u]_{1, p}^p +\|u_n-u\|_{p, \beta}^p\right)\geq C\|u_n-u \|_{1, p, \beta}^p,
\end{split}
\end{equation}and the conclusion holds.

For the singular case, we proceed as follows. Observe that by H\"{o}lder inequality and \eqref{ineq singular},
\begin{equation*}
\begin{split}
\int_{\mathbb{R}^N}|\nabla u_n(x)-\nabla u(x)|^p\,dx& = \int_{\mathbb{R}^N}\dfrac{|\nabla u_n(x)-\nabla u(x)|^p}{(|\nabla u_n(x)|+|\nabla u(x)|)^{\frac{(2-p)p}{2}}}(|\nabla u_n(x)|+|\nabla u(x)|)^{\frac{(2-p)p}{2}}\,dx\\ & \leq \left(\int_{\mathbb{R}^N}\dfrac{|\nabla u_n(x)-\nabla u(x)|^2}{(|\nabla u_n(x)|+|\nabla u(x)|)^{2-p}}\,dx\right)^{p/2}\left(\int_{\mathbb{R}^N}(|\nabla u_n(x)|+|\nabla u(x)|)^{p}\,dx\right)^{1/p}\\&\leq C\int_{\mathbb{R}^N}\left( \Phi_p(\nabla u_n(x))-\Phi_p(\nabla u(x))\right)(\nabla u_n(x)-\nabla u(x))\,dx,
\end{split}
\end{equation*}where we have used that $u_n$ is bounded. A similar inequality holds for $[u_n-u]_{s, p}$ when $s\in (0, 1)$. Thus, considering the cases $\gamma=0$ and $\gamma\in (0, 1]$ as before, a similar argument to \eqref{similar i} gives the conclusion in the case $1<p<2$.

\end{proof}

\section{Boundedness of solutions}\label{B}
In this section, we prove that any radial solution 
u of problem \eqref{eq}  is bounded.

\begin{proof}[Proof of Theorem \ref{teo_boundedness}] In what follows, we will apply a De Giorgi's iteration scheme to control the level sets of a solution $u$ to problem \eqref{eq}.

 Let $u\in \mathcal{W}^{s, p}_\beta(\mathbb{R}^N)$ be a solution of problem \eqref{eq}. For a positive integer $k$, define
$$w_{k}:= (u-(1-2^{-k}))_{+}.$$Then, as in \cite{FP}, the following holds
\begin{equation}\label{property w}
w_{k+1} \leq w_k \text{ in }\mathbb{R}^N, \quad u(x)< (2^{k+1}-1)w_k \text{ in }\left\lbrace w_{k+1}>0\right\rbrace, \quad \text{and }\left\lbrace w_{k+1}>0\right\rbrace \subset \left\lbrace w_{k}>2^{-(k+1)}\right\rbrace. 
\end{equation} Also, $w_k(x)\to (u(x)-1)_{+}$ a.e. in $\mathbb{R}^N$, so by Fatou Lemma
\begin{equation}\label{convergence to u}
\int_{\mathbb{R}^N} |x|^\alpha |(u-1)_{+}|^q\,dx \leq \liminf_{k\to \infty}\int_{\mathbb{R}^N} |x|^{\alpha}|w_k|^q\,dx.
\end{equation}
Now, letting
$$C(\gamma)=\begin{cases}1, \quad \text{if } \gamma=0, 1,\\
\max\left\lbrace \dfrac{1}{\gamma}, \dfrac{1}{1-\gamma} \right\rbrace,  \quad \text{if } \gamma \in (0, 1), \end{cases}$$we get

\begin{equation}\label{estimate w k}
\begin{split}
\|w_{k+1}\|_{s, p, \beta}^p &\leq C( \gamma)\bigg(\gamma\int_{\mathbb{R}^N}|\nabla w_{k+1}|^{p-2}\nabla w_{k+1}\cdot \nabla w_{k+1}\,dx \\ &+(1-\gamma)\int_{\mathbb{R}^N}\int_{\mathbb{R}^N}\dfrac{|w_{k+1}(x)-w_{k+1}(y)|^{p-2}(w_{k+1}(x)-w_{k+1}(y))(w_{k+1}(x)-w_{k+1}(y))}{|x-y|^{N+sp}}\,dx\,dy \bigg)\\& + \int_{\mathbb{R}^N}|x|^{\beta}|w_{k+1}|^{p-1}w_{k+1}\,dx \\ & \leq C(\gamma)\bigg(\gamma\int_{\mathbb{R}^N}|\nabla u|^{p-2}\nabla u\cdot \nabla w_{k+1}\,dx\\& +(1-\gamma)\int_{\mathbb{R}^N}\int_{\mathbb{R}^N}\dfrac{|u(x)-u(y)|^{p-2}(u(x)-u(y))(w_{k+1}(x)-w_{k+1}(y))}{|x-y|^{N+sp}}\,dx\,dy \bigg)\\& + \int_{\mathbb{R}^N}|x|^{\beta}|u|^{p-1}w_{k+1}\,dx \\ & = C(\gamma)\int_{\mathbb{R}^N}|x|^{\alpha}|u|^{q-1}w_{k+1}\,dx,
\end{split}
\end{equation}where we have used the fact that
$$|v_+(x)-v_+(y)|^p \leq (v(x)-v(y))(v_+(x)-v_+(y))|v_+(x)-v_+(y)|^{p-2}$$and also that $u$ is a solution of \eqref{eq}. Now, by \eqref{property w}, we get 
\begin{equation}\label{estimate w k 2}
\int_{\mathbb{R}^N}|x|^{\alpha}|u|^{q-1}w_{k+1}\,dx \leq \int_{\mathbb{R}^N}|x|^{\alpha}|(2^{k+1}-1)w_k|^{q-1}(2^{k+1}-1)w_k\,dx\leq 2^{q(k+1)}\int_{\mathbb{R}^N}|x|^\alpha |w_{k}|^q\,dx.
\end{equation}By Lemma \ref{continuity} applied to $w_{k+1}\in \mathcal{W}^{s,p}_\beta(\mathbb{R}^N)$ (recall also the notation \eqref{exponents e}) and by \eqref{estimate w k} and \eqref{estimate w k 2},
\begin{equation}
\begin{split}
\int_{\mathbb{R}^N}|x|^\alpha |w_{k+1}|^q\,dx &\leq C\|w_{k+1}\|_{s, p, \beta}^{\frac{pe_2+qe_1}{e_1+e_2}}\\& \leq\left(2^{q(k+1)}C(\gamma)\int_{\mathbb{R}^N}|x|^\alpha |w_{k}|^q\,dx \right)^{\frac{pe_2+qe_1}{pe_1+pe_2}}\\& \leq M^k\left(\int_{\mathbb{R}^N}|x|^\alpha |w_{k}|^q\,dx \right)^{1+\delta}
\end{split}
\end{equation}for some $M>1$ large enough and for $\delta:= (q-p)e_1/(pe_1+pe_2)>0$ since $q>p$. Hence, letting 
$$a_k:=\int_{\mathbb{R}^N}|x|^\alpha |w_{k}|^q\,dx,$$we have obtained
$$a_{k+1}\leq M^ka_k^{1+\delta}, \quad \delta>0.$$Thus, choosing
$$a_0\leq M^{-\frac{1}{•\delta^2}}$$as in \cite{FP}, gives $a_k\to 0$ as $k\to \infty$.  Therefore, we obtain that there is $\varepsilon\in (0, 1)$ such that if
 \begin{equation}\label{assumption with eps}
 \int_{\mathbb{R}^N}|x|^\alpha |w_{0}|^q\,dx <\varepsilon
 \end{equation}then
 $$\lim_{k\to \infty}\int_{\mathbb{R}^N}|x|^\alpha |w_{k}|^q\,dx =0.$$Hence, \eqref{convergence to u} implies  $\|u\|_{L^{\infty}(\mathbb{R}^N)}\leq 1$. Observe that 
 $$\int_{\mathbb{R}^N}|x|^\alpha |w_{0}|^q\,dx=\int_{\mathbb{R}^N} |x|^\alpha u_{+}^q\,dx$$so the result is valid under the assumption
 $$\int_{\mathbb{R}^N} u_{+}^q|x|^\alpha\,dx<\varepsilon.$$If now
 $\xi=\int_{\mathbb{R}^N} u_{+}^q|x|^\alpha\,dx>0$ is arbitrary, then we choose  $C>1$ large enough so that
 \begin{equation}
\int_{\mathbb{R}^N} \left(\frac{u_{+}}{C}\right)^q|x|^\alpha\,dx<\varepsilon.
 \end{equation}Moreover, $u/C$ also satisfies \eqref{estimate w k} with an extra constant at the end. Hence, from the above argument 
 $$\|u/C\|_{L^{\infty}(\mathbb{R}^N)}\leq 1 \text{ or }\|u\|_{L^{\infty}(\mathbb{R}^N)}\leq C.$$This ends the proof of the theorem.
\end{proof}

\section{Non-existence of solutions} \label{sec.6}

We begin this section with the following computation, which justifies our use of the non-existence result \cite[Proposition 1.4]{ROS}, originally stated for bounded domains. Given $\lambda>1$, we can write
$$
u(\lambda x)-u(x)=\int_1^\lambda \frac{d}{dt}u(tx)\,dt= \int_1^\lambda \nabla u(tx)\cdot x\,dt.
$$
Therefore, 
$$
\left| \frac{u(\lambda x)-u(x)}{\lambda -1} \right|^r  \leq  \frac{1}{(\lambda-1)^r} \left(\int_1^\lambda |\nabla u(tx)| |x|\,dt \right)^r.
$$
Using Jensen's inequality yields
$$
\left| \frac{u(\lambda x)-u(x)}{\lambda -1} \right|^r  \leq  \frac{1}{\lambda-1} \int_1^\lambda |\nabla u(tx)|^r |x|^r\,dt.
$$
Then, integrating on $\mathbb{R}^N$, using Fubini's Theorem and changing variables we get
\begin{align} \label{eq.lam}
\begin{split}
\int_{\mathbb{R}^N}\left| \frac{u(\lambda x)-u(x)}{\lambda -1} \right|^r\,dx  &\leq  \frac{1}{\lambda-1} \int_1^\lambda \frac{1}{t^{N+r}} \int_{\mathbb{R}^N }|\nabla u(y)|^r |y|^r\,dy dt\\
&=\frac{1}{N+r-1}\frac{1-\lambda^{-(N+r-1)}}{\lambda-1} \int_{\mathbb{R}^N }|\nabla u(y)|^r |y|^r\,dy \\
&\leq
C_{N,r} \int_{\mathbb{R}^N } | \nabla u(y)|^r |y|^r\,dy,
\end{split}
\end{align}
since $\lambda >1$ and $N+r-1>0$. Therefore, if $|\nabla u(x)||x|\in L^r(\mathbb{R}^N)$, for some $r>1$, \eqref{eq.lam} enable us to apply  \cite[Lemma 4.2]{ROS}, and hence \cite[Proposition 1.4]{ROS} holds with $\Omega=\mathbb{R}^N$.

\begin{proof}[Proof of Theorem \ref{teo.no.exist}]
Let $\gamma\in (0,1]$  and $\beta\in\mathbb{R}$. Given  $u\in W^{1,p}_{rad, \beta}(\mathbb{R}^n)$, denote $u_\lambda(x)=u(\lambda x)$ for $\lambda>1$.  An easy computation gives that, when $\beta>-p$, 
$$
\|u_\lambda\|_{1,p,\beta} = \|u_\lambda\|_{p,\beta} + [u_\lambda]_{1,p} \leq \lambda^{-\frac{ N+\beta}{p}} \|u\|_{p,\beta} + \lambda^{-\frac{N-p}{p}}[u]_{1,p}  \leq \lambda^{-\tau}\|u\|_{1,p,\beta},
$$
where $\tau :=\frac{N-p}{p}$.

Let $f(x,t)=|x|^\alpha |t|^{q-2}t$ and $F(x,u)=\int_0^u f(x,t)\,dt=\frac{1}{q} |x|^\alpha |u|^q$, being $u\in W^{1,p}_{rad, \beta}(\mathbb{R}^N)$  a weak solution of \eqref{eq}. It is straightforward to see that
\begin{equation} \label{eq.1.3}
\tau t f(x,t) > N F(x,t) + x \cdot F_x(x,t) \quad \text{ for all } t\in \mathbb{R}^N, t\neq 0 
\end{equation}
whenever $q> \frac{N+\alpha}{\tau}$, that is, when $q>\frac{p(N+\alpha)}{N-p}$. Moreover, recall that $u\in L^\infty(\mathbb{R}^N)$ by definition of the space $ \mathcal{\tilde X}(\mathbb{R}^N)$.  So,  by \cite[Proposition 1.4]{ROS} we have that $u\equiv  0$.

When $\gamma=0$ and $s\in (0,1)$, an analogous computation gives that any $u\in W^{s,p}_\beta(\mathbb{R}^N)$ fulfills that $\|u_\lambda\|_{s,p,\beta}\leq \lambda^{-\tau}\|u\|_{s,p,\beta}$ with $\tau = \frac{N-sp}{p}$ when  $\beta>-sp$. Therefore, \eqref{eq.1.3} holds for $q>\frac{p(N+\alpha)}{N-sp}$, giving again that $u\equiv 0$.
\end{proof}

\section*{Acknowledgment}
We would like to thank the anonymous referees who thoroughly read this article. Their precise and detailed suggestions have significantly contributed to its improvement and clarity.

\section*{Appendix}
The following variant of the Mountain-Pass Theorem will be of use for our arguments. See \cite[Theorem 10]{BN}. 
\begin{lemma}[Brezis--Nirenberg] \label{lema.BN}
Let $E$ be a Banach space and $\Phi \in C^{1}(E,\mathbb{R})$ satisfy $\Phi(0)=0$ and the following conditions:
\begin{itemize}
    \item[(i)] there exists $r>0$ such that $\displaystyle \inf_{\|u\|=r}\Phi(u)>0$;
    \item[(ii)] there exists $v$ with $\|v\|>r$ such that $\Phi(v)<0$;
    \item[(iii)] the Palais--Smale condition: if $\Phi'(u_n)\to 0$ and $\Phi(u_n)\to c$,  
    then $(u_n)$ has a subsequence converging to $u\in E$;
    \item[(iv)] there exists a continuous mapping $P:E\to E$ such that  
    $P(0)=0$, $P(v)=v$, and $\Phi(P(u))\le \Phi(u)$ for every $u\in E$.
\end{itemize}
Then there exists a critical point $u^{*}\in \overline{P(E)}$ such that
\[
\Phi(u^{*}) \ge \inf_{\|u\|=r} \Phi(u).
\]
\end{lemma}

\end{document}